\documentclass{amsart}
\usepackage{amsfonts}
\usepackage{amssymb}
\usepackage{amsthm}

\newtheorem{introthm}{Theorem}
\newtheorem{theorem}{Theorem}[section]
\newtheorem{lemma}[theorem]{Lemma}

\theoremstyle{definition}

\theoremstyle{remark}

\newcommand{\gen}[1]{\langle#1\rangle}

\begin{document}
\title[Conjugacy class graph]{Triangle-free cyclic conjugacy class graph of a finite group}

\author[Lewis]{Mark L. Lewis}
\address{Department of Mathematical Sciences, Kent State University, Kent, OH  44242, U.S.A.}
\email{lewis@math.kent.edu}

\author[Mohammadian]{Abbas Mohammadian}
\address{Department of Pure Mathematics, Ferdowsi University of Mashhad, \linebreak \\ P.O.Box 1159-91775, Mashhad, Iran.\\}
\email{abbasmohammadian1248@gmail.com}


\begin{abstract}
We generalize the enhanced power graph by replacing elements with conjugacy classes.  The main result of this paper is to determine when this graph is triangle-free.
\end{abstract}

\maketitle


\section{Introduction}

Throughout this paper, all groups are finite.  One recent trend in group is to study the interaction between various graphs and groups.  The excellent expository paper by Cameron \cite{camex} outlines much of the work in this area and we highly recommend that the reader consult this work.

Perhaps the graph that has received the most attention is the commuting graph.  Given a group $G$, its {\it commuting graph} is the graph with vertices $G \setminus \{ 1 \}$ and there is an edge between $x, y \in G \setminus \{ 1 \}$ when $xy = yx$.  Herzog, Longobardi and Maj proposed a modification to this graph in \cite{mh-pl-mm} by defining the \textit{commuting conjugacy class graph} (or CCC-graph) of a group $G$ as the graph with vertex set, the set of nontrivial conjugacy classes of $G$, and two distinct vertices $x^{G}$ and $y^{G}$ are adjacent when $\langle x^{\prime },y^{\prime}\rangle$ is abelian for some $x^{\prime }\in x^{G}$ and $y^{\prime } \in y^{G}$.  In the paper \cite{mh-pl-mm}, they consider the connectivity and the diameter of the corresponding connected components of the CCC-graph of $G$.  In the paper \cite{MEFW}, the second author with others characterize the groups whose CCC-graphs are triangle-free, and that paper serves as motivation for this paper. 

In this paper, we make a similar modification to a graph that was originally studied under the name cyclic graph (see \cite{imper} and \cite{imle}), but in the literature it is most often called the enhanced power graph (see \cite{AACNS}, \cite{BeBhu}, \cite{BeKiMu}, and \cite{BeDe}).  There is an expository paper on enhanced power graphs with much of the known results at \cite{enhex}.  We define the {\it enhanced power graph} (or cyclic graph) of $G$ to be the graph whose vertex set is $G \setminus \{ 1 \}$ with an edge between $x,y \in G \setminus \{ 1 \}$ when $\langle x, y \rangle$ is cyclic.  Following the idea of the CCC-graph, define the {\it cylic conjugacy class graph} of $G$ to be the graph with vertex set, the set of nontrivial conjugacy classes of $G$, so that there is an edge between distinct $x^G$ and $y^G$ when there exist $x' \in x^G$ and $y' \in y^G$ so that $\langle x', y' \rangle$ is cyclic.  We will use $\Delta (G)$ to denote the CCC-graph of $G$.  

This graph is studied in \cite{MGEL}.  In that paper, we consider the connectivity and diameter of this graph.  We determine the universal vertices of this graph and when this graph is a complete graph.  We also determine when this graph is an empty graph.   


The aim of this paper is to classify all finite groups $G$ with a triangle-free cyclic conjugacy class graph. We first have the case where $|G|$ is odd, and we prove:

\begin{introthm}\label{intro-|G|=odd} 
If $G$ is a group of odd order and $\Delta (G)$ is triangle-free then either (1) $G$ is a $3$-group of exponent $3$ or (2) $G$ is a Frobenius group of order $3 \cdot 7^a$ for a positive integer $a$ where the Frobenius kernel $N$ has exponent $7$.
\end{introthm}

When $|G|$ is even, we consider the case that $G$ is a $2$-group.

\begin{introthm}\label{intro-Characterize 2-group}
If $G$ is a $2$-group and $\Delta (G)$ is triangle-free, then $G$ has exponent at most $4$ and every element of order $4$ is conjugate to its inverse.
\end{introthm}

We next study the case that $G$ is solvable, $2$ divides $|G|$, and $Z(G) = 1$.

\begin{introthm} \label{intro-Solvable} 
If $G$ is a solvable group with $|G|$ even, $Z (G) = 1$, and $\Delta (G)$ is triangle-free, then either (1) $G$ is a Frobenius group where the Frobenius kernel is an elementary abelian $3$-group or $5$-group and a Frobenius complement is $Z_2$ or $Q_8$, (2) $G$ is a Frobenius group where the Frobenius kernel has exponent $4$ and every element of order $4$ is conjugate to its inverse and a Frobenius complement is $Z_3$, or (3) $G$ is a $2$-Frobenius group that is a $\{ 2, 3 \}$-group, a Sylow $3$-subgroup has order $3$, a Sylow $2$-subgroup has exponent $4$ and every element of order $4$ is conjugate in $G$ to its inverse, and if $N$ is the Fitting subgroup of $G$, then $G/N \cong S_3$.
\end{introthm}

Finally, we are left with the case that $G$ is nonsolvable.

\begin{introthm} \label{intro-Non-soluble} 
If $G$ be a centerless non-solvable group and $\Delta (G)$ is triangle-free, then $G$ is isomorphic to either (1) $PSL(2,q)$ where $q \in \{4, 7, 9 \}$, (2) $PSL(3,4)$, or (3) $G/N \cong {\rm PSL} (2,4)$ and $N$ an elementary $2$-group that is isomorphic to copies of the natural module for $ {\rm PSL} (2,4)$. 
\end{introthm}


\section{Results}

We begin with a technical lemma that shows that if $G$ has a CCC graph $\Delta (G)$ which is triangle-free, then there are strong restrictions on the elements of $G$.  We will see that the observation that every nonidentity element has order that either a prime or the square of a prime is key to classifying groups where $\Delta (G)$ is triangle-free.

\begin{lemma} \label{Element orders}
Suppose $G$ is a group such that $\Delta (G)$ is triangle free.  For any element $1 \ne x \in G$, the number of distinct conjugacy classes of $G$ intersecting $\left\langle x\right\rangle -\{1\}$ is at most two and $N_{G} (\left\langle x \right\rangle )/C_{G} (\left\langle x \right\rangle )$ can have at most two orbits on its action on $\left\langle x\right\rangle -\{1\}$.  In particular, the order of every nontrivial element of $G$ is either prime or a square of a prime.
\end{lemma}

\begin{proof} 
Suppose $1 \ne x \in G$.  The the conclusion that the number of distinct conjugacy classes intersecting $\langle x \rangle$ is an immediate consequence of the fact that $\Delta (G)$ is triangle-free.  Observe that any two generators of $\left\langle x\right\rangle$ are conjugate in $G$ if and only if they are conjugate in $N_{G} (\left\langle x \right\rangle )$.  Thus, it follows that $N_{G} (\left\langle x \right\rangle )/C_{G} (\left\langle x \right\rangle )$ has at most two orbits on its action on $\left\langle x\right\rangle -\{1\}$.  Suppose now that there exists $x \in G \setminus \{ 1 \}$ such that $|x| = p^{\alpha} q^{\beta}$ for distinct primes $p$ and $q$.  Then the vertices $x^{G}$, $(x^{p^{\alpha }})^{G}$, and $(x^{q^{\beta }})^{G}$ give rise to a triangle, which is a contradiction.  Similarly, if $|x| = p^{k}$ for some $k\geq 3$, then the vertices $x^{G}$, $(x^{p})^{G}$ and $(x^{p^{2}})^{G}$ induce a triangle, again a contradiction.  Therefore, we conclude that the order of every nontrivial element is either prime or a square of a prime.
\end{proof}

We now consider groups $G$ of odd order where $\Delta (G)$ is triangle-free.  We begin to characterize these groups.

\begin{lemma}\label{Group exponent} 
Suppose that $G$ is a group so that $\Delta (G)$ is triangle-free and $|G|$ is odd. Then the nontrivial elements of $G$ have prime order.  Moreover, either $Z(G) = 1$ or $G$ is a $p$-group and $\exp (G) = p$.
\end{lemma}

\begin{proof}
Consider an element $x \in G \setminus \{1\}$.  By Lemma \ref{Element orders}, $|x| = p$ or $p^{2}$ for some prime $p$. Suppose that $|x|=p^{2}$.  Since $x^{G} \neq (x^{-1})^{G}$ (as $G$ has odd order), the vertices $x^{G},(x^{-1})^{G}$ and $(x^{p})^{G}$ induce a triangle, which is a contradiction.  Therefore, $|x|=p$ and the first conclusion holds.

To prove the second conclusion, suppose that $Z(G) \neq 1$. We can find an element $1 \ne z \in Z(G)$ so that $z$ is an element of prime order $p$.  If another prime $q$ divides $|G|$, then $G$ would have an element of order $pq$ and this contradicts Lemma \ref{Element orders}.  Therefore, $G$ is a $p$-group.
\end{proof}

We are now able to classify the groups $G$ of odd order where $\Delta (G)$ is triangle-free.  This theorem includes Theorem \ref{intro-|G|=odd}.

\begin{theorem}\label{|G|=odd} 
Let $G$ be a group of odd order. Then $\Delta (G)$ is triangle-free if and only if one of the following cases occurs:
\begin{itemize}
\item[(1)] $G$ is a $3$-group of exponent $3$;
\item[(2)] $G$ is a Frobenius group of order $3 \cdot 7^a$ for a positive integer $a$ where the Frobenius kernel $N$ has exponent $7$ and for every nontrivial element $1 \ne x \in N$, the normalizer $N_G(\langle x \rangle)$ contains a Frobenius complement of $G$ which has order $3$.
\end{itemize}
\end{theorem}

\begin{proof}
Suppose first that $\Delta (G)$ is triangle-free.  Assume first that $G$ is a $p$-group.  We know that we have $1 \ne z \in Z(G)$, so $\langle z \rangle$ intersects $p-1$ conjugacy classes of $G$, and by Lemma \ref{Element orders}, we have $p - 1 = 2$; so $p = 3$.  Applying Lemma \ref{Group exponent}, $G$ has exponent $3$ and we have (1).  

Now, assume $G$ is not a $p$-group.  {}From Lemma \ref{Group exponent}, we know that all elements have prime order.  By the classification of groups with prime order in \cite{CDLW}, we see that $G$ is a Frobenius group with a Frobenius complement of prime order $q$.  Let $Q$ be a Frobenius complement, and we have $Q = \langle x \rangle$ for some element $x$.  Observe that each nontrivial element of  $\langle x \rangle$ lies in its own conjugacy class of $G$ and so, $\langle x \rangle $ intersects nontrivially $q-1$ conjugacy classes of $G$.   Using Lemma \ref{Element orders}, we have $q-1 = 2$, so $q = 3$.  

We have that $N$ is the Frobenius kernel of $G$.   Let $1 \ne z \in Z(N)$ be an element of order $p$ for some prime $p$ (note that $p \ge 5$).  We see that either $H$ is contained in $N_G (\langle z \rangle)$ or $H \cap N_G (\langle z \rangle) = 1$.  Observe that if $H \cap N_G(\langle z \rangle) = 1$, then $N_G (\langle z \rangle) = C_G (\langle z \rangle)$ and $N_G (\langle z \rangle)/C_G (\langle z \rangle) = 1$ has $p-1$ orbits on $\langle z \rangle$ which contradicts Lemma \ref{Element orders}.  Thus, we have $H \le N_G (\langle z \rangle)$.  Since $H$ does not centralize $z$, we have $H \cap C_G (\langle z \rangle) = 1$, and so, $N_G (\langle z \rangle)/C_G (\langle z \rangle) \cong Z_3$ has $(p-1)/3$ orbits on $\langle z \rangle$.  By Lemma \ref{Element orders}, we know that $(p-1)/3 \le 2$.  Since $p \ge 5$, we conclude that $p = 7$.  Therefore, $N$ is a $7$-group.  Since all elements of $N$ have prime order (i.e. order $7$), we apply Lemma \ref{Group exponent} to conclude that $N$ has exponent $7$.  Let $1 \ne x \in N$.  By Lemma \ref{Element orders}, we know that $N_G (\langle x \rangle)/C_G(\langle x \rangle)$ has at most two orbits that intersect $\langle x \rangle$ nontrivially.  Hence, $N_G (\langle x \rangle)$ contains a Frobenius complement of $G$.  This proves (2).

Conversely, suppose that $G$ is a $3$-group of exponent $3$.  Let $C_1$, $C_2$, and $C_3$ be three nontrivial conjugacy classes of $G$, and let $x_i$ be a representative of $C_i$.  We know that $x_i$ has order $3$.  Thus, $C_i$ is adjacent to $C_j$ if and only if $\langle x_i \rangle$ is conjugate to $\langle x_j \rangle$.  If $C_i \ne C_j$, then this occurs if and only if $x_j$ is conjugate to $(x_i)^2$.  Hence, there can only be edge between at most two of distinct $C_1$, $C_2$, and $C_3$.  We conclude that $\Delta (G)$ has no triangles.

Now, suppose that $G$ is a Frobenius group of order $3 \cdot 7^a$ whose Frobenius kernel has exponent $7$ and for every nontrivial element $1 \ne x \in N$, the normalizer $N_G(\langle x \rangle)$ contains a Frobenius complement of $G$.  Let $N$ be the Frobenius kernel and let $H$ be a Frobenius complement.  Take $C_1$, $C_2$, and $C_3$ be three distinct nontrivial conjugacy classes of $G$, and let $x_i$ be a representative of $C_i$.  We know that $x_i$ has order $3$ or $7$.  Suppose $C_i \ne C_j$.  If $x_i$ and $x_j$ have different orders, then since $G$ is a Frobenius group, they will not commute, and so they will not be adjacent in $\Delta (G)$.  Hence, $C_1$, $C_2$, and $C_3$ form a triangle only if $x_1$, $x_2$, and $x_3$ have they same order.  If this order is $3$, they must all intersect $H$ which is not possible since $H$ only has two nonidentity elements.  Thus, they must have order $7$.  For $C_1$, $C_2$, and $C_3$ to form a triangle, they all must intersect nontrivially some subgroup $\langle x \rangle$ where $x$ has order $7$.  Since $N_G (\langle x \rangle)$ contains a Frobenius complement, we see that $\langle x \rangle$ has only two orbits under the action of $N_G (\langle x \rangle)/C_G (\langle x \rangle)$ and so, at most two of $C_1$, $C_2$, and $C_3$ can intersect $\langle x \rangle$ nontrivially, which is a contradiction.  Therefore, we conclude that $\Delta (G)$ has no triangles.  
\end{proof}

\section{Groups of even order}

In this section, we classify the groups $G$ with even orders where $\Delta (G)$ is triangle-free.  Assume $Z(G) \neq 1$. Since elements of $G$ have prime power orders, then $G$ is a $p$-group.  Since $|G|$ is even, it follows that $G$ is a $2$-group. 

\begin{lemma} \label{Z(G) >1 implies 2}
Suppose $G$ is a group with $\Delta (G)$ is triangle-free, $|G|$ is even, and $Z(G) > 1$, then $G$ is a $2$-group.
\end{lemma}

\begin{proof}
Since $|G|$ is even, we know that $G$ has an element of order $2$.  We claim that $Z(G)$ contains an element of order $2$.  Thus, we can find $x \in G$ and $z \in Z(G)$ so that $x$ has order $2$ and $z$ has order $3$.  (If $o (z) > 3$, then $\Delta (G)$ contains a triangle.)  Observe that $x^G$, $z^G$, $(z^2)^G$ forms a triangle.  Thus, $Z(G)$ contains an element of order $2$.  A similar argument shows that if $G$ contains an element whose order is a prime greater than $2$, then $\Delta (G)$ will contain a triangle.  
\end{proof}

We now characterize the $2$-groups where $\Delta (G)$ is triangle-free.  We say an element $g$ of a group $G$ is {\it real} if $g$ is a conjugate to $g^{-1}$.  We say $G$ is a {\it real group} if all of its elements are real.  Please note that $G$ is a real group if and only if all of its irreducible characters are real valued (see Problem 2.11 of \cite{text}).  Note that this next theorem includes Theorem \ref{intro-Characterize 2-group}.

\begin{theorem}\label{Characterize 2-group}
Suppose $G$ is a $2$-group.  Then the following are equivalent:
\begin{enumerate}
\item $\Delta (G)$ is triangle-free
\item $G$ has exponent at most $4$ and every element of order $4$ is conjugate to its inverse.
\item $G$ has exponent at most $4$ and $G$ is a real group.
\end{enumerate}  
In particular, if $\Delta (G)$ is triangle-free and $G$ is a $2$-group, then $Z(G)$ is elementary abelian.
\end{theorem}

\begin{proof}
Suppose $G$ is a $2$-group and $\Delta (G)$ is triangle-free.  Suppose that $x \in G$ has order greater than $4$.  Then $x^G$, $(x^2)^G$, and $(x^4)^G$ forms a triangle, a contradiction.  Thus, the exponent of $G$ is at most $4$.  If $x$ and $x^{-1}$ lie in different conjugacy classes, then $x^G$, $(x^2)^G$, and $(x^{-1})^G$ form a triangle, a contradiction.  Thus, $x$ and $x^{-1}$ must be conjugate.

Conversely, suppose that $G$ has exponent at most $4$ such every element of order $4$ is conjugate to its inverse.  If $G$ has exponent $2$, then we know $G$ is elementary abelian, and thus, $\Delta (G)$ has no edges, so certainly no triangles.  Thus, we may assume $G$ has exponent $4$.  Let $w^G$, $x^G$, and $y^G$ be three distinct conjugacy classes.  If there is an edge between $x^G$ and $y^G$, then $\langle x, y \rangle$ is cyclic, and so, it has order $4$.  The hypothesis that the elements of order $4$ are conjugate to their inverses imply that one of $x$ and $y$ has order $2$ and the other has order $4$.  Without loss of generality, $x$ has order $2$ and $y$ has order $4$.  If $w$ has order $2$, then clearly, $w^G$ is not adjacent to either $x^G$ or $y^G$, and if $w$ has order $4$, the hypothesis on elements of order $4$ implies that $w^G$ is not adjacent to $y^G$.  In both cases, we see that $x^G$, $y^G$, and $w^G$ do not form a triangle.  Therefore, $\Delta (G)$ is triangle-free.

Suppose $G$ has exponent at most $4$ and every element of order $4$ is conjugate to its inverse.  Let $x$ be a nonidentity element of $G$.  If $x$ has order $2$, then $x = x^{-1}$ and if $x$ has order $4$, we are assuming $x$ is conjugate to $x^{-1}$.  It follows that $x$ is real.  We deduce that all elements of $G$ are real, and so, $G$ is real.  Conversely, if $G$ has exponent $4$ and $G$ is a real group, then every element of order $4$ in $G$ will be conjugate to its inverse.

Suppose $z \in Z(G) \setminus \{ 1 \}$.  Since we must have that $z$ is conjugate to $z^{-1}$, it must be that $z = z^{-1}$ which implies that $z$ has order $2$.  We conclude that $Z(G)$ is elementary abelian.
\end{proof}

We note that all elementary abelian $2$-groups and all extra-special $2$-groups will satisfy the condition of Lemma \ref{Characterize 2-group}.  In particular, the quaternion and dihedral groups of order $8$ will have graphs that are triangle-free.  In addition, all semi-extraspecial $2$-groups will satisfy the conditions of Lemma \ref{Characterize 2-group}.  For examples with nilpotence class $3$, we have SmallGroup (64,23), SmallGroup (128, 1755), SmallGroup (128, 1758), SmallGroup (128, 1759), and SmallGroup (128, 1760).  

We now consider what additional structural information we can obtain about $2$-groups where $\Delta (G)$ is triangle-free.  In \cite{GupNew}, Gupta and Newman show that $2$-groups of exponent $4$ with $m$ generators have nilpotence class at most $3m-2$.   We would not be surprised if one could find a stronger bound on the nilpotence class of $G$, but that seems to be beyond the scope of this paper.  Since the our group is real-valued, we can prove the following regarding the quotient modulo the derived subgroup. 

\begin{lemma}
Let $G$ be a $2$-group where $\Delta (G)$ is triangle-free.  Then $G/G'$ is an elementary abelian $2$-group.
\end{lemma}

\begin{proof}
Suppose that there exists an element $x \in G$ so that $xG'$ has order $4$ in $G/G'$.  Then $x^G, (x^2)^G, (x^3)^G$ induces a triangle in $\Delta (G)$ unless $x^3 = x^g$ for some $g \in G$.  Hence, $x^2 = [x,g] \in G'$ which contradicts our choice of $x$.
\end{proof}

In light of Lemma \ref{Z(G) >1 implies 2}, if $G$ is not $2$-group, then $Z(G) = 1$.  In what follows, the structure of the groups $G$ are determined when $G$ is centerless.   We first consider the solvable case.  This includes Theorem \ref{intro-Solvable}.

\begin{theorem} \label{Solvable} 
If $G$ is a solvable group with $|G|$ even and $Z (G) = 1$, then $\Delta (G)$ is triangle-free if and only if $G$ is isomorphic to one of the following groups:
\begin{enumerate}
\item [(1)] Frobenius group $(Z_3^a) \rtimes Z_2$ where $a$ is a positive integer,
\item [(2)] Frobenius group $(Z_5^a) \rtimes Z_2$ where $a$ is a positive integer,
\item [(3)] Frobenius group $N \rtimes Z_3$ where $N$ has exponent at most $4$ and every element of order $4$ is conjugate in $N$ to its inverse,
\item [(4)] Frobenius group $((Z_3 \times Z_3)^a) \rtimes Q_8$ where $a$ is a positive integer,
\item [(5)] Frobenius group $((Z_5 \times Z_5)^a) \rtimes Q_8$ where $a$ is a positive integer, or
\item [(6)] $2$-Frobenius group where $G$ is a $\{2, 3\}$-group, a Sylow $3$-subgroup has order $3$, a Sylow $2$-subgroup has exponent $4$ and every element of order $4$ is conjugate in $G$ to its inverse, and if $N$ is the Fitting subgroup of $G$, then $N$ is an elementary abelian $2$-group and $G/N \cong S_3$.
\end{enumerate} 
\end{theorem}

\begin{proof}
By \cite{rb} and the fact that the order of any nontrivial element of $G$ is a prime or the square of a prime, we have $|G| = 2^m p^n$, for some odd prime $p$, and nonnegative integers $m$ and $n$.  Now, we consider two cases:
	
Case 1. $G$ has a nontrivial, normal subgroup of odd order. Since $G$ is a solvable $CIT$-group, by \cite[II, Theorem 1]{ms}, $G = P \rtimes T$ is a Frobenius group whose kernel and complement are $P$ and $T$, respectively. Moreover, $P$ is an abelian Sylow $p$-subgroup of $G$ and $T$ is a Sylow $2$-subgroup of $G$. By \cite[10.5.5]{djsr}, $T$ is cyclic or generalized quaternion group. If $|T| = 4$, then $G/P \cong T \cong \mathbb {Z}_{4}$ and $\Delta (G)$ is not triangle-free. Indeed, if $G/P = \langle xP \rangle$, then the three conjugacy classes $x^{G}$, $(x^{-1})^{G}$ and $(x^{2})^{G}$ induce a triangle. Therefore $T \cong \mathbb{Z}_{2}$ or $T \cong Q_{8}$.  This implies that $P$ is abelian.  When $T = Z_2 = \langle z \rangle$ and when $T = Q_8$, take $Z(Q_8) = \langle z \rangle$.  We know that $z$ has order $2$ and inverts every element of $P$.  Thus, if $x \in P$ has order $p^2$, then $\langle x \rangle$ has $(p^2-1)/2$ nontrivial orbits under $\langle z \rangle$.  Since $p^2-1 \ge 4$, this will yield a triangle in $\Delta (G)$ when $T = Z_2$.  Assume for now that $T = Z_2$.  We see that $P$ is elementary abelian.  Now, for $x \in P$ having order $p$, then $\langle x \rangle$ intersects $(p-1)/2$ conjugacy classes of $G$.  Hence, if $p \ge 7$, then $\langle x \rangle$ will intersect at $3$ conjugacy classes, and we have a triangle.  Thus, we conclude that $p = 3$ or $p = 5$.  We see that $G$ is a Frobenius group, either $(Z_3^a) \rtimes Z_2$ or $(Z_5^a) \rtimes Z_2$.   

Suppose now that $T = Q_8$.  We know that an irreducible module for $Q_8$ has the form $Z_p \times Z_p$.  Thus, $\Omega_1 (T) = (Z_p \times Z_p)^a$.  We see that $Q_8$ has $(p^2-1)/8$ orbits on $Z_p \times Z_p$.  When $p \equiv 1 ({\rm mod ~} 4)$, each of these orbits intersects nontrivially two of the subgroups of order $p$ in sets of size $4$.  If $p \ge 13$, then three orbits will intersect nontrivially the same subgroup of order $p$, and we will get a triangle, which is a contradiction.  Thus, when $p \equiv 1 ({\rm mod ~} 4)$, we have $p < 13$, which yields $p = 5$.  On the other hand,  When $p \equiv 3 ({\rm mod ~} 4)$, each of these orbits intersects nontrivially four of the subgroups of order $p$ in sets of size $2$.  If $p \ge 7$, then three orbits will intersect nontrivially the same subgroup of order $p$, and we will get a triangle, which is a contradiction.  Thus, when $p \equiv 3 ({\rm mod ~} 4)$, we have $p < 7$, which yields $p = 3$.  

We know by Lemma \ref{Group exponent}, that $P$ has exponent at most $p^2$.  We have two cases $p = 3$ and $p = 5$.  Suppose first $p = 3$ and $P$ has an element $x$ of order $9$.  We know that the nontrivial elements in $\Omega_1 (\langle x \rangle)$ form one orbit in $G$.  However, the remaining $6$ elements cannot form a single orbit since $P$ is abelian, so the orbits have sizes that divide $8$.  Thus, $\langle x \rangle$ intersects nontrivially at least three conjugacy classes of $G$, and $\Delta (G)$ would have a triangle which is a contradiction.  Similarly, when $p = 5$, we see that $P$ has an element of order $25$.  Again, the nontrivial elements in $\Omega_1 (\langle x \rangle)$ form one orbit in $G$.  However, the remaining $20$ elements cannot form a single orbit since $P$ is abelian, so the orbits have sizes that divide $8$.  Thus, $\langle x \rangle$ intersects nontrivially at least six conjugacy classes of $G$, and $\Delta (G)$ would have a triangle which is a contradiction.  Therefore, $P$ has exponent $p$, and we have the Frobenius groups: either $((Z_3 \times Z_3)^a) \rtimes Q_8$ or $((Z_5 \times Z_5)^a) \rtimes Q_8$.
	
Case 2. $G$ has no normal subgroups of odd order. Let $N = O_{2} (G)$, and let $P$ be a Sylow $p$-subgroup of $G$. Since $P$ acts fixed-point-freely 
on $N$, we may apply Theorem \cite[10.5.5]{djsr}, to see that $P$ is a cyclic group. Hence, $|G| = 2^{m} p^{k}$ where $k=1$ or $2$.  We know by Higman's classification of solvable groups where all elements have prime power order that $G$ is either a Frobenius group or a $2$-Frobenius group.  We have two cases:
	
(I) $G/N$ has odd order.  Notice that $G$ will be a Frobenius group in this case.   Let $g\in G$ be an element of order $p$.  If $p > 3$ or $p = 3$ and $k = 2$, then $\{ g^G, (g^2)^G, (g^3)^G \}$ induces a triangle in $\Delta (G)$, which is a contradiction.  Hence, we have that $p = 3$ and $k = 1$.  Suppose $x \in N$ has order $4$.  If $x$ and $x^{-1}$ are not conjugate in $G$, then $\{ x^G, (x^2)^G, (x^{-1})^G \}$ induces a triangle which is a contradiction.  Thus, we have $g \in G$ so that $x^g = x^{-1}$.  It follows that $(x^{-1})^g = (x^g)^{-1} = x$, and $2$ divides the order of $g$.  Let $g^*$ be the $2$-part of $g$ and observe that $x^{g^*} = x^{-1}$ and $g^* \in N$.  In particular, $x$ is conjugate to its inverse in $N$ as desired.  
	
(II) $G/N$ has even order.  In this case, we see that $G$ is a $2$-Frobenius group.  Since $G/N$ is a solvable $CIT$-group, by \cite[Theorem 1]{ms}, $G/N = K/N\rtimes H/N$ is a Frobenius group such that its Frobenius kernel $K/N$ is the abelian Sylow $p$-subgroup and a Frobenius complement $H/N$ is a Sylow $2$-subgroup of $G/N$.  It is known for $2$-Frobenius groups that $H/N$ and $K/N$ must be cyclic.  Notice that if $4$ divides $|H/N|$, then we can find $h \in H$ so that $h^G, (h^2)^G, (h^3)^G$ form a triangle in $\Delta (G)$ which is a contradiction.  Thus, $H/N \cong Z_2$.  It then follows that $|K/N| = p^k$ and if $\langle x \rangle$ is a Sylow $p$-subgroup of $G$, then $\langle x \rangle$ intersects $(p^k-1)/2$ conjugacy classes of $G$.  If either $p > 3$ or $p = 3$ and $k = 2$, then we obtain a triangle in $\Delta (G)$, a contradiction.  Therefore, we must have $|K/N| = 3$, and so, $G/N \cong S_3$.  By Theorem  1 (3) of \cite{lytkina}, we see that $N$ is an elementary abelian $2$-group.  We know that the exponent of a Sylow $2$-subgroup is $4$.  If we have an element of order $4$ in $G$ that is not conjugate to its inverse, then as we have seen before, we will have a triangle in $\Delta (G)$.

Conversely, suppose $G$ is a Frobenius group $(Z_3^a) \rtimes Z_2$ or $(Z_5^a) \rtimes Z_2$.  We see that every element of $G$ has order $2$ or order $p$ where $p$ is $3$ or $5$.  We see that a Frobenius complement for $G$ will fix every subgroup of order $p$ in the Frobenius kernel.  Thus, each subgroup of order $p$ will intersect two conjugacy classes of $G$.  Hence, we have no triangles in $\Delta (G)$.   

Next, suppose $G$ is a Frobenius group $N \rtimes Z_3$ where $N$ has exponent at most $4$ and every element of order $4$ is conjugate in $N$ to its inverse.  Now, every element of $G$ will have order $2$, $3$, or $4$.  Let $C_1, C_2, C_3$ be three distinct conjugacy classes in $G$, and let $x_i$ be a representative of $C_i$.  Suppose $C_i$ and $C_j$ are adjacent with $i \ne j$.  If $x_i$ has order $3$, then $x_j$ must have order $3$ also.  We see that $C_k$ cannot intersect nontrivially $\langle x_i \rangle$.  We see that $x_i$ and $x_j$ must have orders $2$ or $4$.  Obviously, they cannot both have order $2$, and since elements of order $4$ are conjugate to their inverse, they cannot both have order $4$ and intersect nontrivially the same subgroup of order $4$.  Thus, one must have order $2$ and the other have order $4$.  But it is not possible for $C_k$ to intersect nontrivially $\langle x_i, x_j \rangle$.  Thus, $\Delta (G)$ does not contain any triangles.

Now, suppose $G$ is a Frobenius group $((Z_3 \times Z_3)^a) \rtimes Q_8$ or $((Z_5 \times Z_5)^a) \rtimes Q_8$.  We see that every element of $G$ has order $2$, $4$, or $p$ where $p$ is $3$ or $5$.   We see that when $p = 3$ an subgroup of order $2$ in a Frobenius complement for $G$ will fix every subgroup of order $3$ in the Frobenius kernel and when $p = 5$ an subgroup of order $4$ in a Frobenius complement for $G$ will fix every subgroup of order $5$ in the Frobenius kernel.  Thus, each subgroup of order $p$ will intersect two conjugacy classes of $G$.   We know in $Q_8$ that every element of order $4$ is conjugate to its inverse.  Hence, we have no triangles in $\Delta (G)$.

Finally, suppose we have $2$-Frobenius group where $G$ is a $\{2, 3\}$-group, a Sylow $3$-subgroup has order $3$, a Sylow $2$-subgroup has exponent $4$ and every element of order $4$ is conjugate in $G$ to its inverse, and if $N$ is the Fitting subgroup of $G$, then $G/N \cong S_3$.  One can show that the nontrivial elements of $G$ have orders $2$, $3$, and $4$.  The fact that the elements of order $4$ are conjugate to their inverses imply that $\Delta (G)$ will have no triangles.     
%
%
\end{proof}

Notice that SmallGroup (192,1023) and SmallGroup (192,1025) are examples of Frobenius groups where the Frobenius kernels are semi-extraspecial groups of order $64$ and whose Frobenius complements have order $3$.  Thus, these groups satisfy conclusion (3) where the Frobenius kernel is nonabelian.  Also, SmallGroup (200,44) is an example of a group that satisfies conclusion (5).   The simplest example of a group satisfying conclusion (6) is $S_4$.  SmallGroup (384,20164) gives another example of a group satisfying conclusion (6).

We now consider the nonsolvable case.  We will see that this classification is based on the classification of nonsolvable groups whose elements have prime power order that is found in \cite{rb}.

\begin{theorem} \label{Non-soluble} 
Let $G$ be a centerless non-solvable group.  Then $\Delta (G)$ is triangle-free if and only if $G$ is isomorphic to one of the following groups:
\begin{enumerate}
\item $PSL(2,q)$ where $q \in \{4, 7, 9 \}$,
\item $PSL(3,4)$, or 
\item  $G/N \cong {\rm PSL} (2,4)$ and $N$ an elementary $2$-group that is isomorphic to copies of the natural module for $ {\rm PSL} (2,4)$.
\end{enumerate}  
\end{theorem}

\begin{proof}
Since $Z (G) = 1$ and $\Delta (G)$ is triangle-free, in light of Lemma \ref{Element orders}, the orders of the nontrivial elements of $G$ are either primes or squares of primes.  Hence, by \cite{rb}, we have the following cases:

(1) $G$ is isomorphic to one of the groups ${\rm PSL} (2,q)$, $q\in \{4, 7, 8, 9, 17 \}$, ${\rm PSL} (3,4)$, ${\rm Sz} (8)$, ${\rm Sz} (32)$, or ${\rm M}_{10}$.

The groups ${\rm PSL} (2,17)$ and $M_{10}$ have elements of order $8$, so we see that $G \not\cong {\rm PSL} (2,17)$ or ${\rm M}_{10}$.  In addition, each of the groups ${\rm PSL} (2,8)$, ${\rm Sz} (8)$ and ${\rm Sz} (32)$ has an element $x$ of order $p^{2}$ such that $x^{G} \neq (x^{-1})^{G}$ so that their corresponding cyclic conjugacy classes graph contain a triangle, which is a contradiction.  Thus, $G$ is isomorphic to one of the groups ${\rm PSL} (2,4)$, ${\rm PSL} (2,7)$, ${\rm PSL} (2,9)$ and ${\rm PSL} (3,4)$.

(2) $G$ contains an elementary abelian normal $2$-subgroup $N$ such that $G/N$ is isomorphic to one of the groups ${\rm PSL} (2,4)$, ${\rm PSL} (2,8)$, ${\rm Sz} (8)$, ${\rm Sz} (32)$, and $N$ is a direct product of copies $Z_p \times Z_p$ that are isomorphic to the natural module for $G/N$.

Again, if $G/N$ is one of the groups ${\rm PSL} (2,8)$, ${\rm Sz} (8)$ and ${\rm Sz} (32)$, then $G/N$ has an element $x$ of order $p^{2}$ such that $x^{G} \neq (x^{-1})^{G}$.  Pulling back to elements of $G$,  we will find elements that form a triangle in $\Delta (G)$.  Thus, we are left with $G/N \cong {\rm PSL} (2,4)$ and $N$ an elementary $2$-group that is isomorphic to copies of the natural module for $ {\rm PSL} (2,4)$.

Conversely, it is not difficult to see that if $G$ is one of the groups in the conclusion, then $\Delta (G)$ is triangle-free.
\end{proof}





\end{document}